\documentclass[11pt,a4paper]{article}

\usepackage{amsmath,amssymb,amsthm,amscd}
\usepackage[]{fontenc}
\usepackage{xy}
\usepackage{enumerate}
\usepackage{graphicx}

\newif\ifpdf\ifx\pdfoutput\undefined\pdffalse\else\pdfoutput=1\pdftrue\fi

\xyoption{all}
\newtheorem{thm}{Theorem}[section]
\newtheorem{theorem}[thm]{Theorem}
\newtheorem{corollary}[thm]{Corollary}
\newtheorem{proposition}[thm]{Proposition}
\newtheorem{lemma}[thm]{Lemma}

\newtheorem{conjecture}[thm]{Conjecture}
\theoremstyle{definition}
\newtheorem{assumption}[thm]{Assumption}

\newtheorem{remark}[thm]{Remark}

\newcommand{\nc}{\newcommand}
\nc{\cH}{\mathcal{H}} \nc{\cA}{\mathcal{A}} \nc{\cG}{\mathcal{G}}
\nc{\cC}{\mathcal{C}}
\nc{\cO}{\mathcal{O}}
\nc{\cI}{\mathcal{I}}
\nc{\cB}{\mathcal{B}} \nc{\cY}{\mathcal{Y}} \nc{\cK}{\mathcal{K}}
\nc{\cX}{\mathcal{X}} \nc{\cS}{\mathcal{S}} \nc{\cE}{\mathcal{E}}
\nc{\cF}{\mathcal{F}} \nc{\cZ}{\mathcal{Z}} \nc{\cQ}{\mathcal{Q}}
\nc{\cN}{\mathcal{N}} \nc{\cP}{\mathcal{P}} \nc{\cL}{\mathcal{L}}
\nc{\cM}{\mathcal{M}} \nc{\cR}{\mathcal{R}} \nc{\cT}{\mathcal{T}}
\nc{\cW}{\mathcal{W}} \nc{\cU}{\mathcal{U}} \nc{\cD}{\mathcal{D}}
\nc{\cJ}{\mathcal{J}} \nc{\cV}{\mathcal{V}}
\nc{\fr}{{\rightarrow}}
\nc{\rd}{red.deg}
\newcommand{\Q}{\mathbb{Q}}
\newcommand{\Z}{\mathbb{Z}}
\newcommand{\R}{\mathbb{R}}

\newcommand{\pr}{\mathbb P}

\newcommand{\sym}{\mbox{\upshape{Sym}}}

\newcommand{\rk}{\mbox{\upshape{rank}}}

\newcommand{\pb}{\pr_B(\cE)}

\title{Stability and singularities of relative hypersurfaces}
\author{M.A. Barja \footnote{Partially supported by MINECO-MTM2012-38122-C03-01 and by Generalitat de Catalunya 2005SGR00557} , L. Stoppino \footnote{Partially supported by PRIN 2012 {\em Moduli, strutture geometriche e loro applicazioni}, G.N.S.A.G.A.--I.N.d.A.M., and FAR 2013-2014 Insubria}}

\begin{document}

\maketitle

\begin{abstract}
We study relative hypersurfaces over curves, and prove an instability condition for the fibres. This gives an upper bound on the log canonical threshold of the relative hypersurface. We compare these results with the information that can be derived from Nakayama's Zariski decomposition of effective divisors on relative projective bundles.
\end{abstract}


\section{Introduction and discussion of the results}\label{intro}

We work over the complex field. Let $\cE$ be a vector bundle of rank $r\geq 3$ and degree $d$  on a smooth projective curve $B$ of genus $b$.
Consider the relative projective bundle $\pr:=\pr_B(\cE)$ with its structure morphism
$\pi\colon \pr\longrightarrow B$.
Let $\cO_{\pr}(1)$ be the tautological sheaf.

Let us consider a relative hypersurface $X\subset \pr$. This means for us an
element  of  a linear  system on $\pr$ with positive degree on the fibres. 
Such a  system is of the form $|\cO_{\pr}(k)\otimes \pi^*\cM^{-1}|$, where $k>0$ and  $\cM $ is a line bundle on the base $B$ whose degree we shall denote $y$.
Call $f\colon X\longrightarrow B$ the restriction of the morphism $\pi$ to $X$.
$$
\xymatrix{
X\ar^{j}[r] \ar[d]_f &
\pr_B(\cE)
\ar[dl]^\pi \\
B}
$$
In this paper we investigate the inequalities satisfied by invariants  of this fibration and relate them to the geometry and stability properties of $X$ and of its fibres.

\medskip

First of all we study the $f$-positivity of line bundles on $X$.
Recall the  following definition from \cite[Def. 1.3]{BS3}: given a fibred $n$-dimensional variety $g\colon Y\longrightarrow T$  over a smooth curve $T$, and given a line bundle $\cL$ on $Y$, we say that {\em $\cL$ is $g$-positive} if the following inequality holds
\begin{equation}\label{f-pos}
\cL^n\geq n\frac{\cL_{|F}^{n-1}}{h^0(F, \cL_{|F})}\deg g_*\cL.
\end{equation}

When the fibres are of general type the $g$-positivity of the relative canonical sheaf $\omega_g=\omega_Y\otimes \omega_T^{-1}$ is of particular interest:
\begin{equation}\label{slopeinequality}
 K_g^n\geq n \frac{K_F^{n-1}}{p_g(F)}\deg g_*\omega_g.
\end{equation}
This  is usually called {\em slope inequality}.
It is a classical result that the slope inequality holds for relatively minimal fibred surfaces of genus $\geq 2$: see \cite{BS3} for an account of the proofs.
The general expectation is that the slope inequality holds in higher dimension (see \cite[Sec.4]{BS3} for a detailed discussion):
\begin{conjecture}[\cite{BS3} Conjecture 4.1]\label{conj-slope}
Let  $g\colon Y\longrightarrow T$   be a fibred variety such that the relative canonical sheaf $\omega_g$ is relatively nef and that it is ample on the general fibres, and such that the general fibres have sufficiently mild singularities. Then it satisfies the slope inequality (\ref{slopeinequality}).
\end{conjecture}

In dimension higher than $2$ almost nothing is known.
The most general result is that in the conditions of the conjecture, the slope inequality holds if $g_*\omega_g$ is $\mu$-semistable (see for instance \cite[Cor. 1.1]{BS3}).

\medskip

Let us come back to a relative hypersurface $f\colon X\longrightarrow B$. 
The main result of the first part of the paper staes that we can completely determine the $f$-positivity of any relatively ample line bundle on $X$.
This property is essentially equivalent to the numerical condition $y/k\geq \mu$ relating the ratio between the relative degree  $k$ of $X$ and  the degree $y$ of the line bundle $\cM$ to the  slope of the vector bundle $\mu:=\deg \cE/r$. 
The results  can be summarized as follows (see Theorems \ref{uno} and \ref{due}).

\begin{theorem}\label{main}
With the notations above, let $X\in |\cO_\pr(k)\otimes\pi^*\cM^{-1}|$.
Then the following propositions hold.
\begin{enumerate}
\item suppose that $k>r$ (i.e. $\omega_f$ is relatively very ample).
Then the following  are equivalent:
\begin{enumerate}
\item the slope inequality (\ref{slopeinequality}) holds;
\item $K_f^{r}\geq 0$ and ${\rm deg}f_*\omega_f\geq 0$;
\item  $y/k\leq \mu$.
\end{enumerate}
\item If $k>1$ the line bundle $\cO_X(h)$ is $f$-positive for any $h\geq 1$  if and only if $y/k\leq\mu$.
\end{enumerate}
\end{theorem}
%
The above results are not hard to prove, by computation of intersection theory.

\bigskip

In the second part of this paper, we investigate the meaning of the above results on the geometry of $X$ and of its fibres.
First of all we derive from Theorem \ref{main} some instability and singularity conditions on the fibres and also on the total space $X$.
These results are particularly significant in the light of the  study of  the properties of  big divisors in $\pr$.
From this perspective we see that some results of other nature, such as the Zariski decomposition of pseudoeffective divisors, imply in some particular cases our singularity results. 
Moreover, with these methods we can find examples of the sharpness of the results.

Let us describe these arguments in detail.

In \cite{BS3} we have seen that the known methods to prove $f$-positivity all need to assume some stability condition.
However we do not need any kind of stability assumption for proving the above results for hypersurfaces.
Some of the methods  described in \cite{BS3}, in particular the one due to Cornalba-Harris and Bost  (Theorem \ref{teoCH}) can thus
be used backwards in this context to prove an instability result (Theorem \ref{corCH}):
\begin{theorem}\label{CHqui}
Let $\cE$ be a $\mu$-unstable sheaf. Then, given any relative hypersurface $X\in |\cO_\pr(k)\otimes\pi^*\cM^{-1}|$ with $y/k>\mu$, any fibre of $f\colon X\longrightarrow B$ is Chow unstable with respect to $\cO_F(h)$ for  any $h\geq 1$.
\end{theorem}

In order to understand the interest of Theorem \ref{CHqui}, let us recall some known facts, referring to Section \ref{birational} for a more detailed discussion.
Let $\mu_1$ and $\mu_\ell$ be the first and the last slope associated to the Harder-Narasimhan sequence of $\cE$.
Recall that  if $\cE$ is $\mu$-unstable $\mu_{\ell}<\mu<\mu_1$, otherwise $\ell=1$ and $\mu=\mu_1$.
The Harder--Narashiman slopes govern the shape  of the positive cones  of divisors on $\pr$, as follows.
\begin{itemize}
\item (Miyaoka \cite{miyaoka}) A line bundle $\cO_\pr(k)\otimes \pi^*\cM^{-1}$ is nef if and only if $y/k\leq \mu_\ell$.
\item (Nakayama \cite{nakayama}) A line bundle $\cO_\pr(k)\otimes \pi^*\cM^{-1}$ is pseudoeffective if and only if $y/k \leq \mu_1$.
\end{itemize}
Using a result due to Lee \cite{Lee}, we can can prove via Theorem \ref{CHqui} a  condition on the singularities of the effective divisors in $|\cO_\pr(k)\otimes \pi^*\cM^{-1}|$ with $y/k\geq\mu$ (see Theorem \ref{TEOlct}); this condition involves the log canonical threshold (lct) of the couple $(\Sigma, {X}_{|\Sigma})$, where $\Sigma$ is a general fibre of $\pi$.
\begin{theorem}\label{teolct}
With the above notation, let $X$ be a relative hypersurface $X\in |\cO_{\pr}(k)\otimes \pi^*\cM^{-1}|$.

If $y/k\in (\mu, \mu_1]$ then any fibre $F$ of $f$ is singular, and
\begin{equation}\label{ineqlct}
lct(\Sigma, {X}_{|\Sigma}) < \frac{r}{k}.
\end{equation}
\end{theorem}
Thus we have a more restrictive geometric condition on the effective divisors in linear systems of the form $|\cO_{\pr}(k)\otimes \pi^*\cM^{-1}|$ such that $y/k>\mu$.
In particular we can make some conclusions on the singularities of $X$ itself
obtaining a (partial) Miyaoka-Nakayama type result (Theorem \ref{slc}):
\begin{theorem}\label{SLC}
 If $X\in |\cO_\pr(k)\otimes \pi^*\cM^{-1}| $ is smooth or is such that $lct(\pr, X)\geq r/k$,  then $y/k\leq\mu$.
\end{theorem}
\begin{remark}
Nakayama's result is often used for proving geographical inequalities, for instance in the relative hyperquadric method.
We believe that this strengthening we obtain can be applied to get stronger inequalities.
\end{remark}
There are some aspects worth discussing.
\begin{enumerate}
\item \label{2}It is natural to wonder if in some cases the implication of Theorem \ref{SLC} can be reversed.
Of course, being a statement on singularities, we can only ask  that the converse implication holds for a {\em general } $X$  in the linear system $|\cO_\pr(k)\otimes \pi^*\cM^{-1}| $.
\item \label{3} The inequality (\ref{ineqlct}) is particularly significant for high enough powers of line bundles, e.g. for  $|\cO_{\pr}(mk)\otimes \pi^*\cM^{-m}|$ with $m\gg 0$, and it gives a lower bound on the multiplicities of the {\em unbounded} points of the linear system. Thus one is lead to try to better understand the asymptotical     behaviour of the linear systems on $\pr$. For instance, does it always have a multiple component, or can its general member be irreducible? This is of course related to the study of the fixed locus of $|\cO_{\pr}(mk)\otimes \pi^*\cM^{-m}|$ for $m\gg0$.
\end{enumerate}
We start from the last questions, and try to get a better understanding of the fixed locus of the linear systems. A first simple study of the case when $\cE$ is of rank 2, in section \ref{rk2}, proves to be enlightening. From the Zariski decomposition of divisors in the ruled surface $\pr$ we can deduce a result analogous to Theorems \ref{teolct} and \ref{SLC}.
Using similar arguments, we see that for the case of arbitrary rank, if the Harder-Narasimhan filtration of $\cE$ has length $2$, and its first sheaf is of rank 1, then --applying a Zariski decomposition of the divisors on $X$ obtained in \cite{nakayama}-- we have for {\em any}  effective non nef divisor in $\pr$ a fixed component with a computable multiplicity.
Quite nicely, we verify that this implies precisely the bound of Theorem \ref{teolct} on the log canonical threshold of $X$.

We now push through this investigation, using again the Zariski decomposition of \cite{nakayama}, and study explicitly the schematic fixed locus of any divisor in the case $\ell=2$, thus completely understanding the locus of unbounded points of the big and not nef divisors in $\pr$ in this case.

Let us summarize the results obtained (Proposition \ref{zar} and Proposition \ref{fixed}).

\begin{theorem}
Suppose that $\ell=2$.
Let us consider the linear system $|\cO_\pr(km)\otimes \pi^*\cM^{-m}|$.
Suppose that $y/k>\mu_2$ (i.e. the system is big and not nef) and that $ (\mu_1-\mu_2)$ divides $(y-\mu_2 k)$.
Then the schematic fixed locus of $|\cO_\pr(km)\otimes \pi^*\cM^{-m}|$ contains the codimension $\rk \, \cE_1$ cycle
$$m\frac{y-\mu_2k}{\mu_1-\mu_2}\pr(\cE/\cE_1),$$
and asymptotically coincides with it.
\end{theorem}
\begin{remark}
It is worth remarking here that the precise statement of the theorem above shuold be that {\em the fundamental cycle} of  the schematic fixed locus of the linear system contains the cycle above, but  we think no confusion should arise from this more concise notation.
\end{remark}
This result implies, in the case $\ell=2$, Theorems \ref{teolct} and \ref{SLC}; moreover, it can be used to prove a converse to these theorems for the case $\rk\,\cE_1=1$.
We then can use it to give an answer to the first question above, as follows (see Theorem \ref{teo-rk2}).
\begin{theorem}
Suppose that $\ell=2$. If  $\rk\, \cE_1>1$ strict inequality always holds in Theorems \ref{teolct} and \ref{SLC}.
If $\rk\,\cE_1=1$, for $m\gg0$ a  {\em general} $X_m\in |\cO_\pr(km)\otimes \pi^*\cM^{-m}|$ satisfies equality:
$$
lct(\pr,X_m)=lct(\Sigma, (X_m)_{|\Sigma})= \frac{\mu_1-\mu_2}{m(y-\mu_2k)}.
$$
In particular, we have that $$lct(\pr, X_m)<\frac{r}{mk} \quad
\iff\quad
lct(\Sigma, {X_m}_{|\Sigma}) < \frac{r}{mk}\quad
\iff\quad
\frac{y}{k}>\mu.
$$

\end{theorem}
From this result we deduce in particular  that  inequality (\ref{ineqlct}) of  Theorem \ref{teolct} is sharp (Corollary \ref{sharpness}).

\medskip
With the same approach, we can also treat the case of Harder-Narasimhan sequences of arbitrary length.
We obtain  an analogous result, which involves  the successive slopes $\mu_i$'s of the Harder-Narasimhan sequence of $\cE$
(ref. Theorem \ref{codim-fix}).
Let $|\cO_\pr(k)\otimes \pi^*\cM^{-1}|$ be a big and not nef linear system on $\pr$, and let $j$ be the integer between $1$ and $\ell-1$ such that  $\mu_{j+1}\leq y/k< \mu_{j}$.
\begin{theorem}\label{codim-FIX}
With the notation above, suppose that  $ (\mu_1-\mu_{j+1})$ divides $(y-\mu_{j+1} k)$.
The schematic fixed locus of $|\cO_\pr(k)\otimes \pi^*\cM^{-1}|$ contains the codimension $\rk\, \cE_j$ cycle
$$\frac{y-\mu_{j+1}k}{\mu_1-\mu_{j+1}}\pr(\cE/\cE_j),$$
and asymptotically coincides with it.
\end{theorem}
This implies the following  result on the log canonical threshold of the couple $(\pr,X)$ (Corollary \ref{cor-fix}).
\begin{corollary}\label{COR-fix}
In the situation of Theorem \ref{codim-FIX}, we have that, for any $X\in |\cO_\pr(k)\otimes \pi^*\cM^{-1}|$,
$$
lct(\pr, X)\leq \frac{\mu_1-\mu_{j+1}}{y-k\mu_{j+1}}.
$$
\end{corollary}
In the case $\ell>2$ this result  -although similar to Theorem \ref{CHqui}- does not imply it, neither is implied by it.

\medskip

Theorem \ref{codim-FIX} tells us  that the case of $\rk\, \cE_1=1$ is indeed the only case when we can asimptotically have a fixed component in the linear systems of line bundles on $\pr$. Moreover, we can see that in most cases the general divisor $X$ in the  linear system $ |\cO_\pr(k)\otimes \pi^*\cM^{-1}|$ is irreducible (Proposition \ref{irreducible}).

Using this results we can explicitly compute the movable cone of $\pr$ (Proposition \ref{movable}), thus reproving a particular case of a result due to Fulger and Lehmann \cite[Prop. 7.1]{sti-due}.

\bigskip

It is natural to try and understand the case of  higher codimensional cycles in $\pr$.
In this case the  nef and pseudoeffective cones have been completely computed by Fulger in \cite{fulger}.
In \cite{BShyper} we generalize as far as possible the present results to the case of relative complete intersections.

\subsubsection*{Acknowledgements} We wish to thank Valentina Beorchia and Francesco Zucconi for their kind encouragement. 
We also thak Yongnam Lee for his interest in this work. Moreover, we are grateful to the anonymous referees for having helped us to considerably improve this paper.

\section{Inequalities for invariants of relative hypersurfaces}\label{sec-inequalities}

\begin{assumption}\label{ass}
Let $\cE$ be a rank $r\geq 3$ vector bundle over a smooth curve $B$,  let
$\pr:=\pr_B(\cE)$ be the relative projective bundle (of quotients), with the natural fibration $\pi\colon \pr\longrightarrow B$.
Let $X\subset \pr$ be a relative (possibly singular) hypersurface.
In other words for us $X$ is an effective divisor
in the linear  system of the form $|\cO_\pr(k)\otimes \pi^*\cM^{-1}|$, with $k>0$ and $\cM$ any invertible  sheaf  on $B$, whose degree we denote $y$. Call $f\colon X\longrightarrow B$ the restriction of $\pi$ to $X$.
 Let $\cO_X(1)$ be the sheaf  $j^*\cO_\pr(1)$ on $X$.
 Recall that $ \pi_*\cO_\pr(h)\cong \sym^h\pi_*\cO(1)\cong \sym^h\cE$ (see for instance \cite[Chap.II, Prop.7.11]{Har}).
 \end{assumption}

 \begin{remark}\label{relative-canonical}
 Under these assumptions, recall that the relative canonical sheaf $\omega_\pi=\omega_\pr\otimes \pi^*\omega_B^{-1}$ of $\pr$ is isomorphic to $\cO_\pr(-r)\otimes \pi^*\!\!\det \cE$.
 \end{remark}


\subsection{Slope inequality}
Our first aim is to study wether or not a slope inequality of the form (\ref{slopeinequality})  holds for the fibration $f\colon X\rightarrow B$.
%
%
Let us start with the following observation.
\begin{remark}\label{vanishing}
With the above notation,
if  $h$ is a positive integer, we have the following short exact sequence of sheaves:
\begin{equation}\label{sequenza}
0\rightarrow \pi_*\cO_\pr(h-k)\otimes \cM\rightarrow \pi_*\cO_\pr(h)\rightarrow f_*\cO_X(h)\rightarrow 0.
\end{equation}
In particular if $h<k$, the left hand sheaf vanishes, and we have isomorphisms of sheaves
$$
f_*\cO_X( h)\cong \pi_*\cO_\pr(h)\cong \sym^h\cE.
$$
Indeed, just tensor the sequence
$$
0\rightarrow \cO_\pr(-X)\rightarrow \cO_\pr\rightarrow \cO_X\rightarrow 0
$$
with $\cO_\pr(h)$, then push it forward via $\pi$, and observe that  the sheaf $R^1\pi_*\cO_\pr(h-k)$ vanishes.
\end{remark}

\begin{theorem}\label{uno}
With the notations above, let $X\in |\cO_\pr(k)\otimes \pi^*\cM^{-1}|$. Then the relative canonical sheaf $\omega_f$ is relatively very ample if and only if $k>r$.
Let us suppose that this is the case. The following are equivalent:
\begin{itemize}
\item[(i)] $y/k\leq \mu$ (resp. $y/k< \mu$);
\item[(ii)] the relative canonical invariants $K_f^{r-1}$ and $\deg f_*\omega_f$ are non-negative (resp. strictly positive);
\item[(iii)] the slope inequality (\ref{slopeinequality}) (resp. strict inequality) holds.
\end{itemize}
In particular  if $\cE$ is $\mu$-semistable then any  relative hypersurface $X$ of relative degree  greater or equal to $ r$ satisfies the slope inequality.
\end{theorem}
\begin{proof}
As the relative canonical sheaf of $\pi$ is $\omega_\pi= \cO_\pr(-r)\otimes \pi^*\!\!\det \cE$ (Remark \ref{relative-canonical}), by adjunction we have that
$$K_f\equiv (K_\pi+X)_{|X}\equiv ((k-r)H-(y-d)\Sigma)_{|X}\equiv (k-r)L-(y-d)F$$

\noindent where $H=[\cO_X(1)]$ and $\Sigma$ is the general fibre of $f$.

From this we get immediately that the restriction of $\omega_f$ is ample (and indeed very ample) if and only if $k>r$.
Then we have that
$$
K_f^{r-1}=(k-r)^{r-2}(k-1)\left(kd-ry\right),
$$
while
$$
f_*\omega_f\cong f_*(\cO_X(k-r)\otimes f^*\cD)\cong  \left(f_*\cO_X(k-r)\right)\otimes \cD, $$
where $\cD$ is a line bundle over $B$ of degree $d-y$.
So, by Remark \ref{vanishing} above, we have that
$$f_*\omega_f\cong \sym^{k-r}\cO_X(1) \otimes \cD\cong \sym^{k-r}\cE \otimes \cD,$$
 and we can easily compute its degree as follows.
$$
\begin{array}{l}
\deg f_*\omega_f= \deg \sym^{k-r}\cE+\rk\, (\sym^{k-r}\cE)(d-y)= \\
\\
\binom{k-1}{r}d + \binom{k-1}{r-1}(d-y)=\binom{k-1}{r-1}\left(\frac{kd-ry}{r}\right).\\
\end{array}
$$
\smallskip

\noindent Supposing $k>r$ as in the statement, it is immediate that the relative canonical invariants are non-negative if and only if $y/k\leq \mu$. We have equality
$$
K_f^{r-1}=r(k-1) \frac{(k-r)^{r-2}}{\binom{k-1}{r-1}}\deg f_*\omega_f.
$$
The slope inequality (\ref{slopeinequality}) thus holds if and only if
$$
r(k-1)\frac{(k-r)^{r-2}}{\binom{k-1}{r-1}}\geq (r-1) \frac{K_F^{r-2}}{h^0(F, \omega_F)} =(r-1)k\frac{(k-r)^{r-2}}{\binom{k-1}{r-1}},
$$
equivalently if and only if
$$
\frac{r}{r-1}\geq \frac{k}{k-1},
$$
and this last inequality is satisfied, as $k>r$ by assumption.

As for the last statement, it follows straight away from Nakayama's result \cite{nakayama} mentioned in the Introduction. In particular if $\cE$ is $\mu$-semistable, $\mu=\mu_1=\mu_\ell$ and if a divisor of the form $kH-y\Sigma$ is effective then $y/k\leq\mu$.
\end{proof}

\begin{remark}
The proof of this result is elementary. Nevertheless this is the only known case where Conjecture \ref{conj-slope} is proved to be true in dimension higher then 2. In \cite[Theorem 1.2]{BS2} we proved a much weaker result which much more effort, we can say that there we shoot a fly with a cannon!
\end{remark}

\begin{remark}
It is important to stress that the total space $X$ in Theorem \ref{uno} can be extremely singular: in the sequel we shall indeed investigate its possible singularities. Theorem \ref{uno} provides new results also in dimension $2$.
For example, it proves the slope inequality for families of plane curves which can be extremely singular: in particular  the total space need not even be normal.
\end{remark}

 \subsection{$f$-positivity of $\cO_X(h)$}
After settling the slope inequality, it is natural to investigate the $f$-positivity of the sheaf $\cO_{X}(1)$ and of its powers $\cO_{X}(h)$ for $h\geq 1$.
The answer turns out to be very simple: $f$-positivity holds for any $h>0$ as soon as the inequality $y/k\leq \mu$ is satisfied.

\begin{theorem}\label{due}
With the notations above. If $k\geq 2$, the following statements are equivalent
\begin{itemize}
\item[(1)] $y/k\leq \mu$;
\item[(2)] there exist an $h>0$ such that $\cO_{X}(h)$ is $f$-positive;
\item[(3)] $\cO_{X}(h)$ is $f$-positive, for any $h>0$.
\end{itemize}

If $k=1$ then ${\cal O}_X(1)$ is always $f$-positive and for $h\geq 2$ $\cO_{X}(h)$ is $f$-positive if and only if $y\leq \mu$.
\end{theorem}
Before proving the theorem we  reduce the $f$-positivity to the following numerical inequality. Note that  here we use the standard convention that considers equal to zero a binomial of the form $\binom{n}{m}$ when $n<m$.
\begin{lemma}\label{lemma}
With the above notations, $f$-positivity of $\cO_X(h)$ for $h\geq 1$ is equivalent to the following inequality:
\begin{equation}\label{corretta}
\left[ h\binom{h+r-1}{r-1} - h\binom{h-k+r-1}{r-1}- \binom{h-k+r-1}{r-1}k(r-1)\right] \frac{dk-ry}{r}\geq 0.
\end{equation}
\end{lemma}
\begin{proof}
Let us compute the invariants involved in the  $f$-positivity of $\cO_{X}(h)$.
Let $H$ be the class of $\cO_\pr(1)$ and $H_X$ the class of $\cO_X(1)$. By the intersection theory we have:
$$
\begin{array}{l}
(hH_X)^{r-1}=h^{r-1}(H_{X})^{r-1}=h^{r-1}(H^{r-1}(kH-y\Sigma))=
 h^{r-1}(kd-y).$$
\\
 \\
 H_{F}^{r-2}=H^{r-3}\Sigma(kH-y\Sigma)=k.
\end{array}
$$
Moreover by using the sequence (\ref{sequenza}) given in Remark \ref{vanishing} we have that

$$
\begin{array}{ll}
\rk\, f_*\cO_X(h)=&\rk\, \pi_*\cO_\pr(h)- \rk\, \pi_*\cO_\pr(h-k)=\\
&\\
&= h^0(\pr^{r-1},\cO_{\pr^{r-1}}(h))- h^0(\pr^{r-1},\cO_{\pr^{r-1}}(h-k))=\binom{h+r-1}{r-1}- \binom{h-k+r-1}{r-1},\\
\end{array}
$$
and that
$$
\begin{array}{ll}
 \deg f_*\cO_{X}(h)&= \deg \pi_*\cO_{\pr}(h)- \deg \pi_* \cO_\pr(h-k)\otimes \cM=\\
& \\
& =\deg \sym^h\pi_*\cO_\pr(1) -\deg \sym^{h-k}\pi_*\cO_\pr(1)-y(\rk\,\sym^{h-k}\pi_*\cO_\pr(1))= \\
 &\\
 & =\binom{h+r-1}{r}d - \binom{h-k+r-1}{r}d-  \binom{h-k+r-1}{r-1}y.
 \\
\end{array}
$$
So, $f$-positivity of $\cO_X(h)$ is equivalent to the following inequality
$$
\begin{array}{l}
h\left(\binom{h+r-1}{r-1}-\binom{h-k+r-1}{r-1}\right)(kd-y)\geq \\
\geq \left(\binom{h+r-1}{r}d- \binom{h-k+r-1}{r}d-\binom{h-k+r-1}{r-1}y\right)k(r-1).\\
\end{array}
$$
Making the substitution
$$
\binom{j+r-1}{r}=\binom{j+r-1}{r-1}\frac{j}{r},
$$
for $j=h$ and $j=h-k$ in the above inequality, we get, after a simple computation, inequality (\ref{corretta}).
\end{proof}

\begin{proof}{\it of Theorem \ref{due}}
Let us give a name to the quantity in the square brackets in inequality (\ref{corretta}) above:
$$
\alpha(h,k):=h\binom{h+r-1}{r-1} - h\binom{h-k+r-1}{r-1}- \binom{h-k+r-1}{r-1}k(r-1).
$$
Note first that
$$
\alpha(h,1)= \binom{h+r-2}{r-1}\left(h+r-1 -h-(r-1)\right)=0,
$$
and that for $h <k$ this is just
$$\alpha (h, k) = h\binom{h+ r -1}{r-1} > 0.$$
Thus  we have to show that $\alpha(h,k)$ is positive for any $h\geq k\geq 2$.
By expanding the binomials, we get
\begin{equation*}
\begin{split}
\alpha(h,k)&= \frac{(h+r-1)!-(h-k+r-1)!\left[ (h-1)(h-2)\ldots(h-k+1)(h+(r-1)k)\right]}{(r-1)!(h-1)!}= \\
&=\frac{(h-k+r-1)!}{(r-1)!(h-1)!}\beta(h,k), \\
\end{split}
\end{equation*}
where
$$\beta(h,k):= (h+r-1)(h+r-2)\ldots (h+r-k)- (h-1)\ldots (h-k+1)[h+(r-1)k].$$
Then the proof is concluded if we show the following
\begin{lemma}
For any triple of integers such that $h\geq k\geq 2$, and $r\geq 2$, the following inequality holds:
\begin{equation}\label{ole}
\prod_{i=1}^{k}(h+r-i)\geq \prod_{j=1}^{k}(h-j+1)+ k(r-1) \prod_{l=1}^{k-1}(h-l).
\end{equation}
\end{lemma}
Let us prove the Lemma by induction on $k\geq 2$.
For $k=2$ and any $h\geq 2$, $r\geq 2$ we have
\begin{equation*}
\begin{split}
(h+r-1)(h+r-2)= h(h+r-2)+(r-1)(h+r-2)= \\
=h(h-1) +(r-1)(2h+r-2)\geq h(h-1)+2(r-1)(h-1),
\end{split}
\end{equation*}
Let the inequality  be true for fixed $k$ and any $h\geq k$ and $r\geq 2$.
Let us consider any $h\geq k+1$ and $r\geq 2$.
Let us multiply both terms of  (\ref{ole}) by $(h+r-k-1)$: we get
$$
\prod_{i=1}^{k+1}(h+r-i)\geq (h+r-k-1)\prod_{j=1}^{k}(h-j+1)+ k(r-1) (h+r-k-1)\prod_{l=1}^{k-1}(h-l)=
$$
$$
= \prod_{j=1}^{k+1}(h-j+1)+ (r-1)\prod_{j=1}^{k}(h-j+1) +  k(r-1) \prod_{l=1}^{k}(h-l) + k(r-1)^2 \prod_{l=1}^{k-1}(h-l)\geq
$$
$$
\geq \prod_{j=1}^{k+1}(h-j+1)+ (r-1)\prod_{l=1}^{k}(h-l) +  k(r-1) \prod_{l=1}^{k}(h-l) + k(r-1)^2 \prod_{l=1}^{k-1}(h-l)\geq
$$
$$
\geq \prod_{j=1}^{k+1}(h-j+1)+ (k+1)(r-1)\prod_{l=1}^{k}(h-l).
$$
and the proof is concluded.
\end{proof}

\begin{remark}\label{twist}
Recall that $f$-positivity of any line bundle is equivalent to $f$-positivity of the same line bundle twisted by the pullback of any line bundle over the base (see e.g. \cite[Remark 1.1]{BS3}). Using this fact, we see that  Theorem \ref{uno} is implied by Theorem \ref{due}, because, as already observed, $\omega_f\cong \cO_{X}(k-r)\otimes \pi^*\cD$ where $\cD$ is a line bundle on $B$ of degree $y-d$.
\end{remark}

\begin{remark}
In general the $f$-positivity of a line bundle and the $f$-positivity of its  multiples are not at all equivalent: see \cite{BS3}, Section 1, in particular Remark 1.8.
\end{remark}

\begin{remark}
It is worth noticing that  if we consider the projective bundle $\pr$ itself, then any relatively ample line bundle $\cO_\pr(k)\otimes \pi^*\cM^{-1}$ (so any line bundle of this form with $k>0$) is $\pi$-positive.
Indeed it is immediate to check that inequality (\ref{f-pos}) is an equality with any of these line bundles.
In general it does not seem immediate to relate the $f$-positivity of a polarized variety with the $f$-positivity of an hyperplane section.
\end{remark}

\subsection{The cones of divisors of $\pr_B(\cE)$}\label{birational}
In order to understand the interest of Theorem \ref{CHqui}, let us briefly describe what is known about nef and big divisors of  $\pr=\pb$. See for reference \cite[Sec.3.b]{nakayama} and \cite{fulger}.
Let us consider the Harder-Narasimhan filtration of $\cE$
$$0=\cE_0\subset \cE_1\subset \ldots \subset \cE_l=\cE,$$
and call $\mu_i:= \mu(\cE_i/\cE_{i-1})$, and $\mu:=\mu$.
Recall in particular that
\begin{equation}\label{HN}
\mu_{\ell}< \mu_{\ell-1} <\ldots< \mu_1,
\end{equation}
and that $\mu_{\ell}<\mu<\mu_1$ unless $\cE$ is semistable, in which case $\cE_1= \cE_\ell= \cE$.

Let us consider the picture in the framework of numerical equivalence: although the results are more general, it seems enlightening to consider this setting.
The  N\'eron-Severi $\mathbb \R$-vector space $N^1_\R(\pr)$ of effective $\mathbb R$-divisors modulo numerical equivalence is $2$-dimensional, generated by the class of the tautological line bundle $H=[\cO_{\pr}(1)]$ and the class of a fibre $\Sigma$.
The sequence of slopes of Harder-Narashiman provide a decomposition of the space $N^1(\pr)$ in subcones as in figure (\ref{coni}).

\begin{figure}[!ht]\label{coni}
\begin{center}
\includegraphics[height=12.25cm]{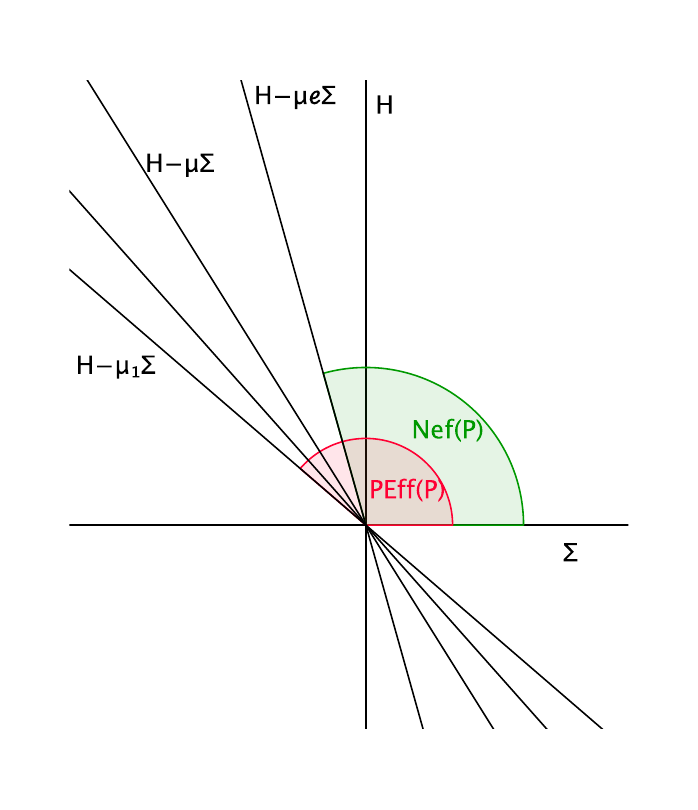}
\end{center}
\caption{Cones of divisors of $\pr$}
\end{figure}

The following classical results give characterizations of  the biggest and smallest subcones:
\begin{itemize}
\item (Miyaoka \cite{miyaoka}) ${\rm Nef}(\pr)=\R^+[ H-\mu_{\ell}\Sigma]\oplus \R^+ [\Sigma]$;
\item (Nakayama \cite{nakayama}) $\overline{{\rm Eff}(\pr)}=\R^+[ H-\mu_{1}\Sigma]\oplus \R^+[\Sigma]$.
\end{itemize}
It should also be mentioned that Wolfe \cite{wolfe} and Chen \cite{chen} proved that the volume function is a polynomial of degree $r$ when it is restricted to the subcones $\R^+ [H-\mu_i\Sigma]\oplus \R^+[H-\mu_{i+1} \Sigma]$, for $i=0,\ldots, \ell-1$, thus having a behaviour similar to the case of surfaces \cite{BKS}.
In case $\cE$ is $\mu$-semistable, all the above cones coincide, while in the  $\mu$-unstable case, this richer birational geometry appears, and can be studied.

\begin{remark}
The results in the previous section thus regard the divisors whose numerical classes belong to the ``intermediate'' cone spanned by $H-\mu_1\Sigma$ and $H-\mu\Sigma$ in this picture. In particular,
Theorem \ref{due} tells us that its complementary cone $\R^+[ H-\mu\Sigma]\oplus \R^+[\Sigma]$  is precisely the cone of numerical classes of  hypersurfaces $X$  which satisfy the $f$-positivity with $\cO_X(h)$ for any $h$.
It is worth making a couple  of remarks.
\begin{itemize}
\item The class $r(H-\mu \Sigma)=rH-d\Sigma$ is the anti-canonical one (Remark \ref{relative-canonical}). It is nef and pseudoeffective, but not big if $\cE$ is $\mu$-semistable, while it is big and not nef if $\cE$ is $\mu$-unstable.
\item Consider the positive cone $P(\pb)$ to be the half-cone of $\{D\in N^1(\pb)\mid D^r\geq 0\}$ that contains the ample cone. This is precisely $\R^+[ H-\mu\Sigma]\oplus \R^+[\Sigma]$.
Indeed, given $a, b\in \mathbb R^{>0}$, we have that $(aH-b\Sigma)^{r}=a^rH^r-ra^{r-1}bH^{r-1}\Sigma=a^rd-ra^{r-1}b\geq 0$ if and only if $b/a\leq d/r=\mu$.
\end{itemize}
\end{remark}

In the rest of the paper we will investigate --using the results of Section \ref{sec-inequalities} and also other methods-- the geometry of the divisors in the linear systems contained in the subcone $$\R^+ [H-\mu_1 \Sigma]\oplus \R^+[H-\mu\Sigma]$$  in case $\cE$ is not $\mu$-semistable. Note that, by Miyaoka's and Nakayama's Theorems above,  the divisors in the interior of this cone are all big and not nef.


\section{Instability results and applications to geometry}
\subsection{An instability result}
Let us start by recalling the following result, which we state in a version the most convenient for our purposes. It has been originally proved by Cornalba and Harris in \cite{CH}, and extended by Bost \cite{bost}- see also \cite{BS3}.

Let $Y$ be a subvariety of dimension $s$ and of degree $k$ in $ \pr^n$. Consider the set $Z(Y)$ of all the $(n-s-1)$- dimensional projective subspaces $L$ in $\pr^n$ that intersects $Y$.
This is an hypersurface of degree $k$ in the Grassmannian $G:=Gr(n - s, n + 1)$, called the {\em Chow variety}.
It is defined by the vanishing of some polynomial of degree $k$ in the Grassmann coordinate ring, which is unique up to a constant factor. This element is called the Chow form of $Y$. The  variety $Y$ is called Chow semi-stable (resp. Chow stable) if its Chow form is semi-stable (resp. stable) for the natural $SL(n + 1)$-action.

\begin{theorem}[Cornalba-Harris, Bost]\label{teoCH}
Let $Y$ be an $n$-dimensional variety with a flat proper surjective morphism $g\colon Y\longrightarrow B$ over a smooth curve $B$.
Let $\cL$ be a line bundle over $X$ which is relatively ample with respect to $g$.
If the general fibre of $g$ is Chow semistable with respect to the immersion induced by the fibre of $\cL$ then $\cL$ is $f$-positive.
\end{theorem}

Applying Theorem \ref{due} and this result, we can immediately deduce the following instability result.
\begin{theorem}\label{corCH}
Let $\cE$ be a $\mu$-unstable locally free sheaf over $B$. Let  X $\in|\cO_{\pr}(k)\otimes \pi^*\cM^{-1}|$ be any relative hypersurface free from vertical components with $y/k>\mu$. Then any fiber $F$ of $f\colon X\longrightarrow B$ is Chow unstable with respect to $\cO_F(h)$ for  any $ h\geq 1$.
\end{theorem}

\begin{remark}Theorem \ref{corCH} above can be extended to any relatively ample effective divisor with $y/k>\mu$. Indeed, suppose $X$ has some vertical component, so $X=X'+\pi^*D$ where $D$ is an effective divisor on $B$. Let $D$ be maximal, so that $X'$ has no vertical components. The divisors $X'$ is in the system $|\cO(k)\otimes \pi^*M^{-1}\otimes \pi^*\cD^{-1}|$, so we can apply Theorem \ref{corCH} to $X'$ because $(y+e)/k>y/k$. Of course the general fibres of $X$ and of $X'$ coincide. \end{remark}

\begin{remark}
It is natural to wonder wether or not the instability of the couple $(F, \cO_F(h))$ implies  instability of the whole $(X, \cO_X(h))$.
This facts seem to be in general not related.
In our situation we will be able to derive asymptotical instability of $(X, \cO_X(h))$, but not directly from the instability of $(F, \cO_F(h))$. We will prove that the fibres need to be highly singular, and from this we get a condition on the  singularities of $X$.
Now from this follows the instability of $(X, \cO_X(h))$: see Theorem \ref{slc}.
\end{remark}

\subsection{A condition on the singularities}

Let us turn our attention on the possible singularities of the relative hypersurfaces in $\pr$.
First we make  a couple of easy remarks. As usual $X\in |\cO_{\pr}(k)\otimes \pi^*\cM^{-1}|$, with $\cM$ line bundle of degree $y$ over $B$.
\begin{itemize}
\item if $y/k<\mu_\ell$, then the numerical class of $\cO_{\pr}(k)\otimes \pi^*\cM^{-1}$ is in the interior of the nef cone, hence this line bundle is ample, and a general member $X\in |\cO_{\pr}(mk)\otimes \pi^*\cM^{-m}|$, for $m\gg0$, will be smooth.
\item If $k>r$ and $y/k>\mu$ and $X\in |\cO_{\pr}(k)\otimes \pi^*\cM^{-1}|$ (so in particular $\cE$ is $\mu$-unstable), by Theorem \ref{uno} we have that $\omega_f$ is relatively ample and  that the relative canonical invariants are negative. When the general fibre is of general type, both $\omega_f$ and $f_*\omega_f$ are nef if $X$ is normal with canonical singularities (see Theorem 1.4 in \cite{ohno} and \cite{fujita}). So we can deduce that in this case $X$ must be singular.
\end{itemize}

We now see that, using the instability result on the fibres of $f$, we can be much more precise as to the singularities that the hypersurface $X$ necessarily carries when we are in the subcone $\R^+[H-\mu_{1}\Sigma]\oplus \R^+[H-\mu \Sigma]$.
First of all we observe that $X$ need to have a singular locus transverse to $f$:

\begin{proposition}\label{prop-sing}
Let $\cE$ be an unstable bundle.
Let $X$ be a relative hypersurface $X\in |\cO_{\pr}(k)\otimes \pi^*\cM^{-1}|$.

If $y/k>\mu$ then any fibre of $f$ is singular.
\end{proposition}
\begin{proof}
By Theorem \ref{corCH} any fibre $F$ of $f$ is Chow unstable with respect to $\cO_F(1)$. But a smooth hypersurface of any degree in $\pr^{r-1}$ is Chow semistable with this sheaf: see for instance \cite[Chap IV, sec. 2]{FM}.
\end{proof}
We now use Theorem \ref{corCH} combined with a result of Lee \cite{Lee} relating the Chow stability of a variety with the log canonical threshold ($lct$) of its Chow variety.

Recall the definition of log canonical threshold (see \cite{KM} for reference).
Let $(Y, \Delta) $ be a pair, with $Y$ normal $\Q$-Gorenstein variety and $\Delta$  a $\Q$-Cartier, $\Q$-divisor on $Y$.
Given any birational morphism $\varphi\colon T\longrightarrow Y$ with $T$ normal, we have
$$
K_T + \varphi^{-1}_*\Delta \equiv \varphi^*(K_Y+\Delta) + \sum a( E_i, Y, \Delta) E_i,
$$
where the $ \varphi^{-1}_*\Delta $ is the strict transform of $\Delta$ and the $E_i$'s are the exceptional irreducible divisors associated to $\varphi$.
Then we define the {\em discrepancy} of the couple  $discrep(Y, \Delta)$ to be the infimum of the $a(E, Y, \Delta)$, taken for any birational morphism $\varphi$ and any exceptional irreducible divisor.
The couple $(Y, \Delta)$ is said to be {\em log canonical} (l.c.) if $discrep(Y, \Delta)\geq -1$.
The log canonical threshold is
$$
lct (Y, \Delta):= \sup \{t>0 | (Y, t\Delta) \mbox{ is log canonical}\}.
$$
This defines a rational number which lies in the interval $(0, 1]$.
The log canonical threshold  is a measure of the singularities of the couple $(Y, \Delta)$: for instance if $Y$ is smooth and  $\Delta$ is reduced and normal crossing, then $lct(Y,\Delta)=1$.

In \cite{Lee}, Lee proved a beautiful condition for a variety to be Chow semistable in terms of the log canonical threshold of its Chow form.
\begin{theorem}[Lee]\label{thmLEE}
Let $Y$ be an $s$-dimensional variety together with a non-degenerate degree $k$ immersion in $\pr^n$.
Let $Z(Y)\subset G$ be the corresponding Chow variety in the Grassmanian $G:= \mbox{Gr}(n-s-1, \pr^n)$.
Suppose that  the following inequality holds
\begin{equation}\label{cond-lct}
lct(G, Z(Y)) \geq \frac{n+1}{k} \quad \mbox{(resp. }> ).
\end{equation}
Then $Y\subset \pr^n$ is Chow semistable (resp.  Chow stable).
\end{theorem}

From this it is immediate to derive the following result, which is a strengthening of Proposition \ref{prop-sing} for the cases $k\geq r$.

\begin{theorem}\label{TEOlct}
Let $\cE$ be an unstable bundle.
Let $X$ be a relative hypersurface $X\in |\cO_{\pr}(k)\otimes \pi^*\cM^{-1}|$.

If $y/k>\mu$ then
\begin{equation}\label{inequality-lct}
lct(\Sigma, {X}_{|\Sigma}) < \frac{r}{k},
\end{equation}
for any fibre $\Sigma$ of $\pi$.
\end{theorem}
\begin{proof}
 For $ k \geq  r$ we just can use LeeÕs Theorem on any fibre. 
For $k<r$ we have that inequality (\ref{inequality-lct}) trivially holds because the log canonical threshold is always smaller or equal than one.
\end{proof}

\begin{remark}
There are (at least) two questions that naturally arise from the above result. First of all one would like to know what singularity result hold for the total space $X$. This question has an easy answer which we discuss in \ref{sing-total}.
Secondarily, observe that inequality (\ref{inequality-lct}) is meaningful for $k\geq r$, and of course it becomes strong for $k\gg r$.
Thus it is natural to try and read it in the framework of the asymptotical behavior of the linear systems $ |\cO_{\pr}(k)\otimes \pi^*\cM^{-1}|$. In \ref{asympt} we begin the discussion of this aspect, and we will pursue this analysis through the rest of the paper.
\end{remark}

\subsection{Log canonical threshold of $X$}\label{sing-total}
From the condition on the log canonical threshold of the fibres of $f$ we can derive a singularity result for $X$ itself, as follows.

\begin{theorem}\label{slc}
Let $\cE$ be an unstable bundle. Let $X$ be a relative hypersurface $X\in |\cO_{\pr}(k)\otimes \pi^*\cM^{-1}|$.

If $y/k>\mu$ then $lct(\pr, X)<r/k$. In particular, if $k\geq  r$, the couple $(\pr, X)$ is not log canonical.
\end{theorem}
\begin{proof}
By a Bertini type result (see \cite[Prop.7.7]{Kollar}) we have that, as the fibres of $\pi$ move in a  free algebraic system, for a general fibre $\Sigma$ we have
$$
discrep( \Sigma, X_{|\Sigma})\geq discrep (\pr, X).
$$
Hence the set $\{t >0 | (\pr, tX) \mbox{ is l.c.}\}$ is contained in the set $\{t >0 | (\Sigma, t(X_{|\Sigma})) \mbox{ is l.c.}\}$, and thus we deduce that the log canonical threshold of the couple $(\pr, X)$ is smaller or equal to the one of $(\Sigma, X_{|\Sigma})$. If $y/k>\mu$, by Theorem \ref{TEOlct}, we have $lct(\Sigma, X_{|\Sigma})<r/k$, so we are done.
\end{proof}

\begin{remark}
Using also the much more subtle Inversion of Adjunction Theorem (see e.g. \cite[Sec.7]{Kollar}) we can say that there is an open set $\cU\subseteq B$ such that $lct (\pi^*(\cU), X\cap \pi^*(\cU))$ is precisely equal to  $lct(\Sigma, X_{|\Sigma})$.
\end{remark}

\begin{remark}
Recall that the numerical class of the anticanonical divisor of $\pr$ is $-K_\pr\equiv rH-(d+2b-2)\Sigma$, where $b$ is the genus of $B$ (Remark \ref{relative-canonical}).
Hence, in the case when $\cE$ is an unstable bundle over a curve of genus $b\geq 2$, we derive from the theorem above that any member of the anticanonical linear system has worse than canonical singularities.
Of course in this case $-K_\pr$ is big but not even nef, so $\pr$ is not even weak Fano.
Deep theorems due to Shokurov and Kawamata assert that for any Fano variety in dimension $\leq 4$ the general member of the anticanonical class has canonical singularities.
\end{remark}

\begin{remark}
From Theorem \ref{slc}, by using Odaka's result \cite{odaka}, we have that if $y/k>\mu$ and $k\geq r$, then $X$ is Chow unstable with respect to $\cO_X(h)$ for any $h>>0$. We shall denote this by saying that $X$ is {\em asymptotically Chow unstable with respect to $\cO_X(1)$.}

So, from the Chow instability of the fibres of $f$ we derive a condition on the singularities of $X$ which implies its asymptotic Chow instability.
On the other hand we can have an unstable $X$ whose fibres are stable.
Indeed, in \cite{Stoppa-Tenni} some examples are constructed of relative hypersurfaces $X\subset \pr_B(\cE)$ which are smooth and asymptotically Chow unstable with $\cO_X(1)$ (twisted by some line bundle on the base).
As these are smooth varieties, the general fibres are themselves smooth hypersurfaces and thus asymptotically Chow stable with respect to the restriction of the line bundle considered.
\end{remark}

\begin{remark}
Theorem \ref{slc} applies to big linear systems giving a condition of singularity on all of their members.
It is of course interesting to know if the general member is irreducible or not.
In Section \ref{codimension} (Proposition \ref{irreducible}) we see that in most cases the answer is positive due to the codimension of the fixed locus.
Let us now  observe the following.
Suppose that an hypersurface $X$ in the linear system $ |\cO_{\pr}(k)\otimes \pi^*\cM^{-1}|$  with $y/k>\mu$ is reducible, and let $X=X_1+X_2$, with $X_i\in  |\cO_{\pr}(k_i)\otimes \pi^*\cM_i^{-1}|$. Then $k=k_1+k_2$ and $y= y_1+y_2$. If one of the $X_i's$, say $X_1$ is vertical, i.e. $k_1=0$ and $y_i\geq 0$, then we can just remove it, and of course $y_2/k_2=y/k_2\geq y/k>\mu$.
If both $k_i$'s are greater than $0$, then at least one of them has to satisfy $y_i/k_i>\mu$. So we can always find an irreducible hypersurface with $y/k>\mu$.
\end{remark}

\subsection{Asymptotical behavior}\label{asympt}
If one considers a divisor $D$ on a smooth variety $Y$, we have that given any  $m\in \mathbb N^{+}$
$$
lct(Y, m D)= \frac{lct(Y, D)}{m}.
$$
This behaviour can suggest the idea that there can be a fixed horizontal divisorial component of the stable base locus of the line bundle $\cO_{\pr}(k)\otimes \pi^*\cM^{-1}$ when $y/k\in (\mu, \mu_1]$.
We see in \ref{codimension} that this is not necessarily the case, but nevertheless we have a strong information on the asymptotical behavior of the singular locus of the hypersurfaces in the linear system $|\cO_\pr(mk)\otimes \pi^*\cM^{-m}|$, as $m$ grows.

\begin{corollary}\label{cor}
Let $\cE$ be an unstable bundle. Let $k\geq r$ and let $\cM$ be a line bundle on $B$ of degree $y$.
Suppose that  $y/k\in (\mu, \mu_1]$.
Then there is a subvariety $\cZ\subset \pr$, such that $\pi(\cZ)=B$, with the property that  any divisor $Z_m\in |\cO_\pr(mk)\otimes \pi^*\cM^{-m}|$ has  multiplicity at least $(km)/r$ in any point $p\in\cZ$:
$$\mbox{ord}_pZ_m\geq \frac{km}{r}.$$
\end{corollary}
\begin{proof}
We just need to recall that the log canonical threshold of a pair $(Y,\Delta)$ can be defined locally in a neighbour of a point $p$, and that
$$lct(Y,\Delta)=\inf_{p\in Y}lct_p(Y,\Delta).$$
The following inequality holds (see for instance \cite[Property 1.14]{mustata}):
$$ lct_p(Y,\Delta)\geq \frac{1}{ord_p(\Delta)}.$$
So from Theorem \ref{slc}, and from Inversion of Adjunction we have that, for any $Z_m\in |\cO_\pr(mk)\otimes \pi^*\cM^{-m}|$  the locus
$$\cZ_m:=\{ p\in \pr \,| \,ord_p(Z_m)\geq km/r \}$$
is dominant over $B$.
Of course $\cZ_m$ is a closed subvariety of $\pr$, thus $\pi(\cZ_m)=B$.
Now observe that this locus is contained in the fixed locus of any $|\cO_\pr(mk)\otimes \pi^*\cM^{-m}|$ for any $m$, so it is contained in the stable fixed locus, and we can just drop the dependence from $m$.
\end{proof}
\begin{remark}
Using the standard notations for  the multiplicity of a linear system (\cite[Def. 2.3.11]{Laz-posI}), we can say that $mult_p ||\cO_\pr(k)\otimes \pi^*\cM^{-1}||\geq k/r$.
\end{remark}

All these results are natural in the light of  the following result of Wilson (see \cite[Theorem 2.39]{Laz-posI}):

\begin{theorem}[Wilson]\label{wilson} A divisor $D$ on a smooth projective variety is nef and big if and only if there exists an effective divisor $N$ and an integer $m_0\geq 1$ such that the linear system $|mD-N|$  is base point free for any $m>m_0$.
\end{theorem}

\begin{remark}
The conditions in Theorem \ref{TEOlct} and Corollary \ref{cor} imply that $kH -\pi^*M$ is big but {\em not} nef for $y/k\in (\mu_\ell,\mu_1]$. Hence, applying Wilson's result, we are sure that  the linear system $|\cO_\pr(mk)\otimes \pi^*\cM^{-m}|$ has at least some points with  multiplicity unbounded wit respect to $m$. In Corollary \ref{cor} we give two more pieces of information for the case $y/k\in (\mu, \mu_1]$: that there is an entire horizontal subvariety of unbounded points, and an explicit lower bound on their multiplicity. In the next section we see that in some cases we can extract similar, sometimes stronger, information from directly studying the geometry of the divisors in $\pr$.
\end{remark}

%
%
%
\section{Zariski decomposition, fixed locus and $lct$ of divisors in $\pr$}\label{explicit}
In this section we make a direct study of the fixed locus of the divisors on $\pr$. 

\subsection{The  case $\rk \, \cE=2$}\label{rk2}
Let us treat the case of a vector bundle $\cE$ of rank $2$, which is not covered by the previous results.
We will see that this simple case will turn out to be interesting.
In this case, a relative hypersurface is just a (not necessarily reduced) multisection of the ruled surface $\pr=\pr_B(\cE)$.
The map $f\colon X\longrightarrow B$ is a finite map, not a fibration.
However,  a result analogous to Corollary \ref{corCH} still holds, as we now see (Lemma \ref{rk2}).
Along the way we will introduce some ideas which will be developed and used more systematically in the subsection \ref{higher-rank}.

Let $\cE$ be an unstable bundle of degree $d$ and rank $2$.
The Harder-Narasimhan sequence has of course $3$ pieces
$$
0=\cE_0\subset \cE_1\subset \cE_2=\cE,
$$
and we have that $\mu_1=\deg \cE_1\in \mathbb Z$, that $\mu_1>\mu=d/2$ by definition of $\cE_1$ and  that $\mu_2=d-\mu_1\in \Z$.
Note that $h^0(\pr, \cO_\pr(1)\otimes \pi^*\cE_1^{-1})=1$, and the effective divisor in the linear system $|\cO_\pr(1)\otimes \pi^*\cE_1^{-1}|$ is $\pr(\cE/\cE_1)$, which is the minimal section of the ruled surface $\pr$ (see e.g. \cite[Chap.V, Sec.2]{Har}). We call this divisor $N$.
Observe that any multiple of $N$ is fixed, as for any $m>0$ we have that
$$h^0(\pr,\cO_{\pr}(mN))=h^0(B, \sym^m\cE\otimes \cE_1^{-1})=1.$$
We can compute the Zariski decomposition of the  effective divisor $X$ as follows.
\begin{lemma}\label{rk2lemma}
Suppose that $X\in |\cO_\pr(k)\otimes \pi^*\cM^{-1}|$ as above.
If $y/k\geq \mu_2$, the negative part of its Zariski $\Q$-decomposition is
\begin{equation}\label{ZD}
 \frac{y-k\mu_2}{\mu_1-\mu_2}N,
\end{equation}
where $N=\pr(\cE/\cE_1)$.
\end{lemma}
\begin{proof}
The divisor $N$ is effective, it has only one irreducible component and
$$N^2=  d-2\mu_1<0.$$
Moreover, let $P$ be the $\Q$-divisor $P:= X- \frac{y-k\mu_2}{\mu_1-\mu_2}N$.
The numerical class of $P$ is
$$
P\equiv kH-y\Sigma- \frac{y-k\mu_2}{\mu_1-\mu_2}N=\frac{k\mu_1-y}{\mu_1-\mu_2}\left(H-\mu_2\Sigma \right),
$$
and $$\frac{k\mu_1-y}{\mu_1-\mu_2}\geq 0,$$
by Nakayama's result, so $P$ is nef.
Moreover
$
PN=d-\mu_1-\mu_2=0.
$
Thus the definition of the Zariski decomposition for surfaces is satisfied.
\end{proof}

\begin{remark}
It can very well happen that a line bundle with numerical class $[H-\mu_1\Sigma]$ is not effective.
Consider for instance the following case: let $M$ a divisor on $B$ of degree $ \mu_1$ such that $M\not \sim \cE_1$.
Then the divisor $H- \pi^*M$ is numerically equivalent (but not linearly equivalent) to $H-\pi^*\cE_1$ but it is not effective, because of the unicity of the greatest destabilizing  subsheaf $\cE_1\subseteq \cE$.
However, also in this case the negative part of the Zariski decomposition of the pseudoeffective divisor $H- \pi^*M$ is $ \frac{y-k\mu_2}{\mu_1-\mu_2}N$.
Note that anyhow in this case the numerical class $[H-\mu_1\Sigma]$ {\em is }effective.
For higher ranks of $\cE$ it can also happen that none of the divisors in the numerical class $[H-\mu_1\Sigma]$  are effective, and that not even some power $[m(H-\mu_1\Sigma)] $ contains an effective divisor: see for instance \cite[Example 1.51]{Laz-posI} and \cite[I.10.5]{Har2}.
\end{remark}
\begin{remark}
Let $X$ be a relative hypersurface in $\pr$. 
The restriction of $X$ with the general fibre $\Sigma\cong \pr^1$ is a $0$-cycle in $\pr^1$ of length $k$.
Let us consider the  Chow stability of these objects: such a cycle is Chow semistable if and only if it has only points of multiplicity $\leq k/2$ (\cite[sec.1.7]{Mum}).
\end{remark}

We are ready to prove the following result.
\begin{proposition}\label{rk2prop}
Let $\cE$ be a $\mu$-unstable rank $2$ vector bundle over $B$ and let $X$ be a relative hypersurface $X\in |\cO_\pr(k)\otimes \pi^*\cM^{-1}|$.
Suppose that $ (\mu_1-\mu_2)$ divides $(y-\mu_2k)$. If  $\frac{y}{k}>\mu$, any fibre of $f\colon X\longrightarrow B$ is Chow unstable with respect to $\cO_F(1)$.
\end{proposition}
\begin{proof}
By the decomposition proved in Lemma \ref{rk2lemma} we know that $X$ has a fixed component of multiplicity $\frac{y-\mu_2k}{\mu_1-\mu_2}$.
Thus we see that any fibre $X\cap \Sigma$ contains  the point $\Sigma\cap N$ with multiplicity at least $\frac{y-\mu_2k}{\mu_1-\mu_2}$. But observe that
$$
\frac{y-\mu_2k}{\mu_1-\mu_2}>\frac{k}{2} \Longleftrightarrow \frac{y/k-\mu_2}{\mu_1-\mu_2}>\frac{1}{2}  \Longleftrightarrow \frac{y}{k}>\mu.
$$
\end{proof}


\subsection{The case  $\ell=2$}\label{higher-rank}
It is natural to try and extend  the same reasoning as in \ref{rk2} to higher rank cases.
Let us start with the easiest possible extension: when the first sheaf in the Harder-Narasimhan filtration of $\cE$ has rank $1$, and $\ell=2$.

\begin{proposition}\label{zar}
Let us suppose that $\rk\, \cE_1=1$ and that $\ell=2$.
Let $X$ be a relative hypersurface in the linear system $|\cO_\pr(k)\otimes \pi^*\cM^{-1}|$ with $y/k\geq \mu_2$.
Suppose that $ (\mu_1-\mu_2)$ divides $(y-\mu_2 k)$. Let $N:= \pr(\cE/\cE_1)\in |H-\pi^*\cE_1|$.
Then the divisor $\frac{y-k\mu_2}{\mu_1-\mu_2}N$ belongs to the scheme-theoretic fixed locus of $X$.

In particular, if $y/k\geq \mu$ (resp. $y/k>\mu$) then
$$
lct(\Sigma, {X}_{|\Sigma}) \leq \frac{r}{k} \quad (\mbox{ resp. }<)
$$
\end{proposition}
\begin{proof}
Let us first observe that $N=\pr(\cE/\cE_1)$ is effective and  fixed. Indeed, we have that
$$h^0(\pr , \cO_\pr (mN))=h^0(B, \sym^m\cE\otimes \cE_1^{-m}) =1$$
 for any $m\geq 1$, because the first piece of the Harder-Narasimhan filtration of $\sym^m\cE$ is $\sym^m\cE_1=\cE_1^{m}$.
 In \cite[Chap.IV]{nakayama} it is proved that the $\Q$-divisor $\frac{y-k\mu_2}{\mu_1-\mu_2}N$ is the negative part of the $\sigma$-decomposition of $X$ as a divisor.
By the definition of the $\sigma$-decomposition of $X$ as a divisor \cite[Chap.III, Def. 1.1 and Def.1.12]{nakayama}, we have that, under our divisibility assumption\footnote{It would not be hard, and in some cases would be natural, to extend this discussion to $\Q$-divisors, or even to $\R$-divisors, but it goes beyond our scope, and we rather not introduce the machinery of $\R$-divisors.},  the number $\frac{y-k\mu_2}{\mu_1-\mu_2}$ is the infimum of $mult_N(Y)$ for any $Y\in |\cO_\pr(k)\otimes \pi^*\cM^{-1}|$.
Thus, any $X\in |\cO_\pr(k)\otimes \pi^*\cM^{-1}|$ contains $ \frac{y-k\mu_2}{\mu_1-\mu_2} N$ in its fixed locus.
Hence also $X_{|\Sigma}$ contains the multiple component $ \frac{y-k\mu_2}{\mu_1-\mu_2} N_{|\Sigma}$ for any fibre $\Sigma$, and so
$$
lct(\Sigma,X_{|\Sigma})\leq lct\left(\Sigma, \frac{y-k\mu_2}{\mu_1-\mu_2}N_{|\Sigma}\right)= \frac{\mu_1-\mu_2}{y-k\mu_2}.
$$
The last inequality is due to the fact that $lct(N_{|\Sigma}) =1$ because $N=\pr(\cE/\cE_1)$ is just a projective subbundle of $\pr$.
Now, it is immediate to check that
$$
\frac{\mu_1-\mu_2}{y-k\mu_2}\leq \frac{r}{k}\quad
 \mbox{if and only if } \quad  \frac{ry}{k}\geq \mu_1+(r-1) \mu_2=d.
 $$
\end{proof}

\begin{remark}\label{bpf}
Suppose that the positive part of the $\sigma$-decomposition of $|\cO_\pr(k)\otimes \pi^*\cM^{-1}|$ is base point free.
Note that this is a nef and not ample divisor whose numerical class is $\frac{\mu_1k-y}{\mu_1-\mu_2}(H-\mu_2\Sigma)$.
By Bertini's Theorem if  $ (\mu_1-\mu_2)$ divides $(y-\mu_2 k)$, for a  general $X\in |\cO_\pr(k)\otimes \pi^*\cM^{-1}|$ we have that $X$ has two components: the $\pr^{r-2}$ projective bundle  $N$  with multiplicity $\frac{y-\mu_2 k}{\mu_1-\mu_2}$ plus another smooth component, which we can also assume to meet transversally $N$. We thus  have that, for  a general $X$
$$lct(\pr,X)=lct\left(\pr,\frac{y-\mu_2 k}{\mu_1-\mu_2}N\right)= \frac{\mu_1-\mu_2}{y-k\mu_2}= lct(\Sigma, X_{|\Sigma}).$$
However, it can happen that the boundary class $[H-\mu_\ell\Sigma]$ of the nef cone is not base point free, and not even semiample: some examples can be found in \cite[Sec.2.3.b]{Laz-posI}.

Another example of this issue can be cooked up in the following way.
Let us consider any vector bundle $\cE$, with Harder-Narasimhan sequence of arbitrary length over $B$ such that $\cE/\cE_{\ell-1}$  is such that $\omega^{-1}_{\pr(\cE/\cE_{\ell-1})}$ is not semi-ample; for instance Nakayama in \cite[Theorem 8]{Nak-taut}  gives a characterization of rank 2 vector bundles with this property. Consider a line bundle $\cM$ over $B$ such that:
$\left(\cO_{\pr}(\mu_{\ell-1})\otimes \pi^*\cM^{-1}\right)_{|\pr(\cE/\cE_{\ell-1})}\cong \omega^{-1}_{\pr(\cE/\cE_{\ell-1})}$.
Then the line bundle $\cO_{\pr}(\mu_{\ell-1})\otimes \pi^*\cM^{-1}$ is nef but not semiample.
\end{remark}
With the next result we see that we can bypass the problem outlined in Remark \ref{bpf} by taking $m\gg0$.
\begin{theorem}\label{teo-rk2}
Suppose that $\rk\, \cE_1=1$ and that $\ell=2$.
Then  for $m\gg0$ if  $X$ is a {\em general} relative hypersurface in the linear system  $|\cO_\pr(km)\otimes \pi^*\cM^{-m}|$ (where $\deg \cM=y$), with $y/k> \mu_2$.
We have
$$lct(\pr, X)= lct(\Sigma, X_{|\Sigma})=\frac{\mu_1-\mu_2}{m(y-\mu_2k)}.$$
In particular we have that $$lct(\pr, X)<\frac{r}{mk} \quad
\mbox{if and only if }\quad
lct(\Sigma, {X}_{|\Sigma}) < \frac{r}{mk}\quad
\mbox{if and only if }\quad
\frac{y}{k}>\mu.
$$
\begin{proof}
The positive part of the $\sigma$-decomposition of   $|\cO_\pr(km)\otimes \pi^*\cM^{-m}|$ is a big and nef divisor, as observed in Remark \ref{bpf}, thus by Wilson's result its log canonical threshold is bounded from below by a non-zero constant.
Hence, for $m\gg0$ and $X$ general we have
$$
lct(\pr, X)=lct\left(\Sigma, \frac{y-\mu_2 k}{\mu_1-\mu_2}N\right)= \frac{\mu_1-\mu_2}{m(y-\mu_2k)}.
$$
Then notice that (see the proof of Theorem \ref{slc}) $lct(\pr, X)\leq lct (\Sigma, X_{\Sigma})$, for $\Sigma $ general fibre.
But $\Sigma $ has $N_{|\Sigma}$ as a component of multiplicity $\frac{y-\mu_2 k}{\mu_1-\mu_2}$, so $lct(\Sigma, X_{\Sigma})\leq  \frac{\mu_1-\mu_2}{m(y-\mu_2k)}$, and the proof is concluded.
\end{proof}
\end{theorem}
\begin{remark}
Another way to phrase the above result is by saying that {\em asymptotically} the {\em schematic fixed locus} of the linear system  $|\cO_\pr(km)\otimes \pi^*\cM^{-m}|$ coincides with $ \frac{y-k\mu_2}{\mu_1-\mu_2} N$.
\end{remark}
\begin{corollary}\label{sharpness}
The inequality (\ref{inequality-lct}) of Theorem \ref{TEOlct} is sharp.
\end{corollary}

Using Nakayama's results in \cite{nakayama}, we can  compute completely the schematic fixed locus of the divisors in $\pr$ in case $\ell=2$, for any rank of $\cE_1$.
\begin{proposition}\label{fixed}
Let us suppose that $\rk\, \cE_1>1$ and that $\ell=2$.
Let $X$ be a relative hypersurface in the linear system $|\cO_\pr(k)\otimes \pi^*\cM^{-1}|$ with $y/k\geq \mu_2$.
Suppose that $ (\mu_1-\mu_2)$ divides $(y-\mu_2 k)$.
The schematic fixed locus of $|\cO_\pr(km)\otimes \pi^*\cM^{-m}|$ contains the codimension $\rk\, \cE_1$ cycle
$$m\frac{y-\mu_2k}{\mu_1-\mu_2}\pr(\cE/\cE_1),$$
and asymptotically coincides with it.

This implies in particular  Theorem \ref{TEOlct} with strict inequality.
\begin{proof}
Let $\widetilde \pr\stackrel{\rho}{\longrightarrow }\pr$ be the blow up op $\pr$ along $\pr(\cE/\cE_1)$. Let $X_m\in |\cO_\pr(k)\otimes \pi^*\cM^{-1}|$.
Nakayama \cite[Chap.IV]{nakayama} proves that there is a   $\sigma$-decomposition of $\rho^*(X_m)$ whose negative part is $\frac{y-k\mu_2}{\mu_1-\mu_2} E$, where $E$ is the exceptional divisor of $\rho$.
Then we easily obtain the first statement, and the asymptotic one comes from the same reasoning as in Theorem \ref{teo-rk2}.

As for the last statement, let $\Sigma_1=\Sigma$ be a general fibre of $\pi$, and let $\Sigma_j\subset \Sigma_{j-1}$, $j=2,\ldots ,r-\rk\,\cE_1:=\alpha$ be a general hyperplane section (so $\Sigma_j\cong \pr^{r-j}$).
Iterating  the argument of Theorem \ref{slc}, by Bertini type theorems, we have that
$$lct(\pr, X)\leq lct(\Sigma, X_{|\Sigma})\leq lct(\Sigma_2, X_{|\Sigma_2})\leq \ldots \leq lct(\Sigma_\alpha, X_{|\Sigma_\alpha}).$$
In the assumptions of Proposition \ref{fixed}, we have that $lct(\Sigma_\alpha, X_{|\Sigma_\alpha})$ contains a component of multiplicity $m\frac{y-\mu_2k}{\mu_1-\mu_2}$, and so  by the same remarks of Theorem \ref{teo-rk2} we conclude that
$$lct(\Sigma_\alpha, X_{|\Sigma_\alpha})\leq \frac{\mu_1-\mu_2}{m(y-\mu_2k)}.$$
Now observe that
$$\frac{\mu_1-\mu_2}{m(y-\mu_2k)}\leq \frac{r}{km}\quad \iff \quad \frac{y}{k}\geq \mu_1+(r-1)\mu_2,$$
and of course $\mu= \rk\,\cE_1 \mu_1+(r-\rk\,\cE_1)\mu_2\geq \mu_1+(r-1)\mu_2$ (indeed, strictly greater if and only if  $\rk\,\cE_1>1$).
\end{proof}
\end{proposition}

\subsection{The general case}\label{codimension}
We are now ready to carry on the computation of the schematic fixed loci of any big and not nef divisor in $\pr$ when the  Harder-Narasimhan sequence of $\cE$ is of arbitrary length.

Let us start by introducing the setting needed. For details see \cite[Chap. IV, from sec. 3.2]{nakayama}.
First of all we define a new variety $\pr(\cE_1\subset \ldots \cE_\ell)\stackrel{\rho}{\longrightarrow}\pr $ over $\pr$ by successive blow ups (\cite[IV Lemma 3.4]{nakayama}):
start with $\pr=\pr(\cE_\ell)$, then blow it up along $\pr(\cE_\ell/\cE_{\ell-1})$ and call
$$\rho_{\ell -1}\colon  \pr(\cE_{\ell-1}\subset \cE_{\ell})\longrightarrow \pr(\cE_\ell)$$
the resulting space and map.
Then we consider the strict transform $(\rho_{\ell-1})_*^{-1}(\pr(\cE/\cE_{\ell-2}))$ and blow up $ \pr(\cE_{\ell-1}\subset \cE_{\ell})$ along it forming
$$
\rho_{\ell- 2}\colon  \pr(\cE_{\ell-2}\subset \cE_{\ell-1}\subset \cE_{\ell})\longrightarrow \pr(\cE_{\ell-1}\subset \cE_\ell),
$$
and so on. The last space we obtain is $\pr(\cE_1\subset \ldots \cE_\ell)$, and $\rho$ is the composition $\rho_{1}\circ \ldots \circ \rho_{\ell -1}$.

Define $E_i\subset \pr(\cE_1\subset \ldots \cE_\ell)$ to be the (strict transform of) the exceptional divisor of $\rho_i$, for $i=1,\ldots, \ell-1$.

Let now $|\cO_\pr(k)\otimes \pi^*\cM^{-1}|$ be a linear system on $\pr$, and
define, as in the proof of Proposition 3.10 of \cite[Chap IV]{nakayama},
$$
\alpha_i:=\frac{y-k\mu_{i+1}}{\mu_1-\mu_{i+1}}
$$
Let $j:= max \{i | \alpha_i>0\}$.

\begin{theorem}\label{codim-fix}
Let $|\cO_\pr(k)\otimes \pi^*\cM^{-1}|$ be a big and not nef linear system on $\pr$, and let $j$ be as defined above.
Suppose that  $ (\mu_1-\mu_{j+1})$ divides $(y-\mu_{j+1} k)$.
The schematic fixed locus of $|\cO_\pr(k)\otimes \pi^*\cM^{-1}|$ contains the codimension $\rk\, \cE_j$ cycle
$$\frac{y-\mu_{j+1}k}{\mu_1-\mu_{j+1}}\pr(\cE/\cE_j),$$
and asymptotically coincides with it.
\end{theorem}
\begin{proof}
Let $X$ be a divisor in $|\cO_\pr(k)\otimes \pi^*\cM^{-1}|$.
In \cite[Lemma 3.11 (1)]{nakayama} it is proved that to obtain a Zariski decomposition of $X$ we have to pass to the space $ \pr(\cE_1\subset \ldots \cE_\ell)$, and  that the negative part of this decomposition  is the $\Q$-divisor $N:= \sum_{i=1}^{\ell-1} \alpha_iE_i$.
Thus any divisor in $|\cO_\pr(k)\otimes \pi^*\cM^{-1}|$ contains the subvariety $\pr(\cE/\cE_i)$ with multiplicity $\alpha_i$.
Noting that  $\alpha_j\geq \alpha_i$ and $\pr(\cE/\cE_j)\subseteq \pr(\cE/\cE_j)$  if $i\leq j$, we have that  $\alpha_j\pr(\cE/\cE_j)$ is contained in the fixed locus of $|\cO_\pr(k)\otimes \pi^*\cM^{-1}|$.

As for the asymptotic statement, it follows from the fact that the positive part of this decomposition is nef, as in Theorem \ref{teo-rk2}.
\end{proof}
We thus have a result on the log canonical threshold, using the same reasoning as in Proposition \ref{fixed}.
\begin{corollary}\label{cor-fix}
In the above situation, we have that, for any $X\in |\cO_\pr(k)\otimes \pi^*\cM^{-1}|$,
$$
lct(\pr, X)\leq \frac{\mu_1-\mu_{j+1}}{y-k\mu_{j+1}}.
$$
\end{corollary}
\begin{remark}
It is important to notice that the above result does not imply Theorem \ref{corCH}, nor it is implied from it. Indeed, the bound obtained above,
$ \frac{\mu_1-\mu_{j+1}}{y-k\mu_{j+1}}$, is less or equal to $r/k$ if and only if
$$y/k\geq \mu_1+(r-1) \mu_{j+1}.$$
This last quantity can be smaller or greater than $\mu$, depending on the reciprocal position of $\mu$ and the $\mu_i$'s, which can be almost arbitrary.
\end{remark}

Just by using the information on the codimension of the fixed locus, we can now see that in most cases the general member of the linear systems we are considering are irreducible.
This can be seen as a Miyaoka-Nakayama type result.

\begin{proposition}\label{irreducible}
For $y/k$ big enough, the general member of $ |\cO_\pr(k)\otimes \pi^*\cM^{-1}|$ is irreducible if one of these conditions hold:
\begin{enumerate}
\item[(1)] $\rk\, \cE_2\geq 3$, and $y/k<\mu_2 $;
\item[(2)] $\rk\, \cE_2=2 $, and $y/k<\mu_3 $;
\end{enumerate}
\end{proposition}
\begin{proof}
This result follows from Bertini's Theorem, by observing that in both the cases above, the codimension of the fixed locus of $ |\cO_\pr(k)\otimes \pi^*\cM^{-1}|$  is greater or equal to $3$.
\end{proof}

Using the above observations, we can compute explicitly the Movable cone of $\pr(\cE)$, thus reproving a particular case of a result due to Fulger and Lehmann \cite[Proposition 7.1]{sti-due}. Recall that the Movable cone is the closure of the cone generated by classes of divisors whose base locus has codimension at least 2.

\begin{proposition}\label{movable}
The movable cone of divisors $Mov(\pr(\cE))$ coincides with the pseudoeffective cone $\overline{Eff(\pr(\cE))}$ if and only if $\rk\, \cE_1>1$. If $\rk\,\cE_1=1$, then $Mov(\pr(\cE))= \R^+[H-\mu_2\Sigma]\oplus \R^+[\Sigma]$.
\end{proposition}
\begin{proof}
Straightforward from Theorem \ref{codim-fix}.
\end{proof}



\bigskip
\noindent Miguel \'Angel Barja,\\  Departament de Matem\`atica  Aplicada I,\\ Universitat Polit\`ecnica de Catalunya,\\
 ETSEIB Avda. Diagonal, 08028 Barcelona (Spain).\\
E-mail: \textsl{Miguel.Angel.Barja@upc.edu}

\bigskip
\noindent Lidia Stoppino,\\Dipartimento di Scienza ed Alta Tecnologia,\\ Universit\`a dell'Insubria, \\Via Valleggio 11, 22100, Como (Italy).\\
E-mail: \textsl {lidia.stoppino@uninsubria.it}.

\end{document}